\documentclass[11pt]{amsart}

\usepackage{amsthm,amsfonts,amsmath,verbatim,mathrsfs,amssymb,verbatim,cite}
\usepackage{dsfont,bbm,youngtab}
\usepackage[usenames]{color}
\usepackage{tikz,graphicx}
\usepackage{hyperref}

\numberwithin{equation}{section}
\numberwithin{figure}{section}
\newtheorem{theorem}{Theorem}[section]
\newtheorem{corollary}[theorem]{Corollary}

\newtheorem{lemma}[theorem]{Lemma}

\DeclareMathOperator{\eup}{e}

\DeclareMathOperator{\ch}{ch}
\DeclareMathOperator{\mult}{mult}

\newcommand{\Rat}{\mathbb Q}

\newcommand{\Z}{\mathbb Z}
\newcommand{\NN}{\Z_{+}}

\newcommand{\Symm}{\mathfrak{S}}
\newcommand{\M}{\mathscr{M}}
\newcommand{\abs}[1]{\lvert#1\rvert}
\newcommand{\la}{\lambda}
\newcommand{\La}{\Lambda}

\newcommand{\ip}[2]{\langle#1,#2\rangle}

\newcommand{\A}{\mathrm A}
\newcommand{\B}{\mathrm B}

\newcommand{\C}{\mathrm C}
\newcommand{\D}{\mathrm D}

\newcommand{\qbin}[2]{\genfrac{[}{]}{0pt}{}{#1}{#2}}

\newcommand{\hfrak}{\mathfrak{h}}

\begin{document}

\title[The $\A_{2n}^{(2)}$ Rogers--Ramanujan identities]{The $\boldsymbol{\A_{2n}^{(2)}}$ Rogers--Ramanujan identities}

\author{S.~Ole Warnaar}

\address{School of Mathematics and Physics,
The University of Queensland, Brisbane, QLD 4072, Australia}

\thanks{Work supported by the Australian Research Council}


\begin{abstract}
The famous Rogers--Ramanujan and Andrews--Gordon 
identities are embedded in a doubly-infinite family
of Rogers--Ramanujan-type identities labelled by positive 
integers $m$ and $n$.
For fixed $m$ and $n$ the product side 
corresponds to a specialised character of the 
affine Kac--Moody algebra $\A_{2n}^{(2)}$ at level $m$,
and is expressed as a product of $n^2$ theta functions of modulus 
$2m+2n+1$, or by level-rank duality, as a product of
$m^2$ theta functions.
Rogers--Ramanujan-type identities 
for even moduli, corresponding to the affine Lie algebras
$\C_n^{(1)}$ and $\D_{n+1}^{(2)}$, 
and arbitrary moduli, corresponding to $\A_{n-1}^{(1)}$,
are also proven.
\end{abstract}

\maketitle

\section{Introduction}

The celebrated Rogers--Ramanujan (RR) identities \cite{Rogers94}
\begin{equation}\label{Eq_RR}
1+\sum_{r=0}^{\infty} \frac{q^{r(r+\sigma)}}{(1-q)\cdots(1-q^r)}=
\prod_{j=0}^{\infty} \frac{1}{(1-q^{5j+\sigma+1})(1-q^{5j-\sigma+4})}
\end{equation}
for $\sigma=0,1$ are two of the most important combinatorial 
identities in all of mathematics, with a remarkably wide 
range of applications.
First recognised by MacMahon and Schur as identities for integer partitions
\cite{MacMahon60,Schur17}, they have since been linked to
algebraic geometry \cite{BMS13}, $\mathrm{K}$-theory \cite{DS94},
conformal field theory \cite{BM98,KMM95}, group theory \cite{Fulman00},
Kac--Moody and double affine Hecke algebras 
\cite{CF12,LM78a,LM78b,LW81a,LW81b,LW82,LW84},
knot theory \cite{AD11,Hikami03,Hikami06}, modular forms \cite{BOR08,Ono09},
orthogonal polynomials \cite{AI83,Bressoud81,GIS99},
statistical mechanics \cite{ABF84,Baxter81}, probability \cite{Fulman01} and
transcendental number theory \cite{RS81}.

In 1974 Andrews \cite{Andrews74} extended \eqref{Eq_RR} 
to an infinite family of Rogers--Rama\-nu\-jan-type identities by 
proving that
\begin{equation}\label{Eq_AG}
\sum_{r_1\geq\dots\geq r_m\geq 0}
\frac{q^{r_1^2+\cdots+r_m^2+r_i+\cdots+r_m}}
{(q)_{r_1-r_2}\cdots(q)_{r_{m-1}-r_m}(q)_{r_m}}=
\frac{(q^{2m+3};q^{2m+3})_{\infty}}{(q)_{\infty}}\,
\theta(q^i;q^{2m+3}),
\end{equation}
where $1\leq i\leq m+1$, $(a)_k=(a;q)_k=(1-a)(1-aq)\cdots (1-aq^{k-1})$ 
(for $k\in\{0,1,\dots\}\cup\{\infty\}$) a $q$-shifted factorial
and $\theta(a;q)=(a;q)_{\infty}(q/a;q)_{\infty}$ a modified
theta function.
The identities \eqref{Eq_AG}, which can be viewed as the analytic counterpart 
of Gordon's partition theorem \cite{Gordon61}, are now commonly referred to
as the Andrews--Gordon (AG) identities.

The various Lie-algebraic interpretations of the Rogers--Ramanujan
and Andrews--Gordon identities attach algebras of low
rank to \eqref{Eq_RR} and \eqref{Eq_AG}. For example, from the above-cited
works of Milne, Lepowsky and Wilson 
it follows that they arise as principally 
specialised characters of integrable highest-weight modules 
of the affine Kac--Moody algebra $\A_1^{(1)}$. 
This raises the question as to whether \eqref{Eq_RR} 
and \eqref{Eq_AG} can be embedded in a larger
family of Rogers--Ramanujan-type identities by considering 
specialised characters of an appropriately chosen affine Lie algebra 
$\mathrm{X}_N^{(r)}$ for arbitrary $N$.
In \cite{ASW99} (see also \cite{FFW08,Warnaar06})
some partial results concerning this question 
were obtained, resulting in Rogers--Ramanujan-type 
identities for $\A_2^{(1)}$. Unfortunately, the approach of \cite{ASW99} 
does not in any obvious manner extend to $\A_n^{(1)}$ for all $n$.

In this paper we give a more satisfactory answer to the above
question by proving Rogers--Ramanujan and 
Andrews--Gordon identities for $\A_{2n}^{(2)}$ for arbitrary $n$.
In their most compact form, the sum-sides are expressed in terms of 
Hall--Littlewood polynomials $P_{\la}(x;q)$
evaluated at infinite geometric progressions.

Let $\theta(a_1,\dots,a_k;q)=\theta(a_1;q)\cdots\theta(a_k;q)$ and 
for $\la=(\la_1,\la_2,\dots)$ an integer partition, let 
$\abs{\la}:=\la_1+\la_2+\cdots$, $2\la:=(2\la_1,2\la_2,\dots)$
and $\la'$ the conjugate of $\la$.
For example, if $\la=(5,3,3,1)$ then $\abs{\la}=12$, $2\la=(10,6,6,2)$
and $\la'=(4,3,3,1,1)$.

\begin{theorem}[$\A_{2n}^{(2)}$ RR and AG identities]\label{Thm_Main}
For $m$ and $n$ positive integers let $\kappa=2m+2n+1$. Then
\begin{subequations}\label{Eq_RR-A2n2}
\begin{align}\label{Eq_RR-A2n2a}
\sum_{\substack{\la \\[1pt] \la_1\leq m}}&
q^{\abs{\la}} P_{2\la}\big(1,q,q^2,\dots;q^{2n-1}\big) \\
&=\frac{(q^{\kappa};q^{\kappa})_{\infty}^n}{(q)_{\infty}^n} 
\prod_{i=1}^n  \theta\big(q^{i+m};q^{\kappa}\big)
\prod_{1\leq i<j\leq n} 
\theta\big(q^{j-i},q^{i+j-1};q^{\kappa}\big) \notag \\
&=\frac{(q^{\kappa};q^{\kappa})_{\infty}^m}{(q)_{\infty}^m} 
\prod_{i=1}^m  \theta\big(q^{i+1};q^{\kappa}\big)
\prod_{1\leq i<j\leq m} 
\theta\big(q^{j-i},q^{i+j+1};q^{\kappa}\big) \notag
\intertext{and}
\label{Eq_RR-A2n2b}
\sum_{\substack{\la \\[1pt] \la_1\leq m}}&
q^{2\abs{\la}} P_{2\la}\big(1,q,q^2,\dots;q^{2n-1}\big) \\
&=\frac{(q^{\kappa};q^{\kappa})_{\infty}^n}{(q)_{\infty}^n} 
\prod_{i=1}^n  \theta\big(q^i;q^{\kappa}\big)
\prod_{1\leq i<j\leq n} 
\theta\big(q^{j-i},q^{i+j};q^{\kappa}\big) \notag \\
&=\frac{(q^{\kappa};q^{\kappa})_{\infty}^m}{(q)_{\infty}^m} 
\prod_{i=1}^m  \theta\big(q^i;q^{\kappa}\big)
\prod_{1\leq i<j\leq m} 
\theta\big(q^{j-i},q^{i+j};q^{\kappa}\big). \notag
\end{align}
\end{subequations}
\end{theorem}
We note the beautiful level-rank duality exhibited by the products
on the right, especially those of \eqref{Eq_RR-A2n2b}. 
We also note that for $n=1$ we recover the Rogers--Ramanujan identities 
and the $i=1$ and $m+1$ instances of the Andrews--Gordon identities 
in a representation due to Stembridge \cite{Stembridge90} 
(see also \cite{Fulman00}).
The equivalence with \eqref{Eq_RR} and \eqref{Eq_AG}
follows from the specialisation formula \cite[p. 213]{Macdonald95}
\[
q^{(\sigma+1)\abs{\la}} P_{2\la}(1,q,q^2,\dots;q)=
\prod_{i\geq 1} \frac{q^{r_i(r_i+\sigma)}}{(q)_{r_i-r_{i+1}}},
\qquad r_i:=\la'_i,
\]
and the fact that $\la_1\leq m$ implies that
$\la'_i=r_i=0$ for $i>m$.
As shown in the next section, the more general 
$P_{\la}(1,q,q^2,\dots;q^n)$ is also expressible 
in terms of $q$-shifted factorials, allowing for
a formulation of Theorem~\ref{Thm_Main}
free of Hall--Littlewood polynomials.

\medskip
  
We have also found an even modulus analogue of Theorem~\ref{Thm_Main}.
Surprisingly, the $\sigma=0$ and $\sigma=1$ cases correspond to dual affine 
Lie algebras.

\begin{theorem}[$\C_n^{(1)}$ RR and AG identities]\label{Thm_Main2}
For $m$ and $n$ positive integers let $\kappa=2m+2n+2$. Then
\begin{align}\label{Eq_RR-Cn}
\sum_{\substack{\la \\[1pt] \la_1\leq m}}&
q^{\abs{\la}} P_{2\la}\big(1,q,q^2,\dots;q^{2n}\big) \\
&=\frac{(q^2;q^2)_{\infty}(q^{\kappa/2};q^{\kappa/2})_{\infty}
(q^{\kappa};q^{\kappa})_{\infty}^{n-1}}{(q)_{\infty}^{n+1}}  \notag \\ 
&\qquad\qquad\; \times
\prod_{i=1}^n  \theta\big(q^i;q^{\kappa/2}\big)
\prod_{1\leq i<j\leq n} \theta\big(q^{j-i},q^{i+j};q^{\kappa}\big)
\notag \\
&=\frac{(q^{\kappa};q^{\kappa})_{\infty}^m}{(q)_{\infty}^m} 
\prod_{i=1}^m  \theta\big(q^{i+1};q^{\kappa}\big)
\prod_{1\leq i<j\leq m} 
\theta\big(q^{j-i},q^{i+j+1};q^{\kappa}\big). \notag
\end{align}
\end{theorem}

\begin{theorem}[$\D_{n+1}^{(2)}$ RR and AG identities]\label{Thm_Main3}
For $m$ and $n$ positive integers such that $n\geq 2$ let $\kappa=2m+2n$. Then
\begin{align}\label{Eq_RR-Dn}
\sum_{\substack{\la \\[1pt] \la_1\leq m}}&
q^{2\abs{\la}} P_{2\la}\big(1,q,q^2,\dots;q^{2n-2}\big) \\
&=\frac{(q^{\kappa};q^{\kappa})_{\infty}^n}
{(q^2;q^2)_{\infty}(q)_{\infty}^{n-1}} 
\prod_{1\leq i<j\leq n} \theta\big(q^{j-i},q^{i+j-1};q^{\kappa}\big) \notag \\
&=\frac{(q^{\kappa};q^{\kappa})_{\infty}^m}{(q)_{\infty}^m} 
\prod_{i=1}^m  \theta\big(q^i;q^{\kappa}\big)
\prod_{1\leq i<j\leq m} 
\theta\big(q^{j-i},q^{i+j};q^{\kappa}\big). \notag
\end{align}
\end{theorem}
The $(m,n)=(1,2)$ case of \eqref{Eq_RR-Dn} is equivalent to 
Milne's modulus $6$ Rogers--Ramanujan identity \cite[Theorem 3.26]{Milne94}.

By combining \eqref{Eq_RR-A2n2}--\eqref{Eq_RR-Dn}
we obtain an identity of mixed type.
\begin{corollary}
For $m$ and $n$ positive integers let $\kappa=2m+n+2$. Then
\begin{multline}\label{Eq_mixed}
\sum_{\substack{\la \\[1pt] \la_1\leq m}}
q^{(\sigma+1)\abs{\la}} P_{2\la}\big(1,q,q^2,\dots;q^n\big) \\
=\frac{(q^{\kappa};q^{\kappa})_{\infty}^m}{(q)_{\infty}^m} 
\prod_{i=1}^m  \theta\big(q^{i-\sigma+1};q^{\kappa}\big)
\prod_{1\leq i<j\leq m} 
\theta\big(q^{j-i},q^{i+j-\sigma+1};q^{\kappa}\big),
\end{multline}
where $\sigma=0,1$.
\end{corollary}

Rogers--Ramanujan-type identities for $\A_{n-1}^{(1)}$ also exist, 
although their formulation is perhaps slightly less satisfactory,
involving a limit. 
\begin{theorem}[$\A_{n-1}^{(1)}$ RR and AG identities]\label{Thm_Main4}
For $m$ and $n$ positive integers let $\kappa=m+n$. Then
\begin{align*}
\lim_{r\to\infty}
q^{-m\binom{r}{2}} 
P_{(m^r)}(1,q,q^2,\dots;q^n)
&=\frac{(q^{\kappa};q^{\kappa})_{\infty}^{n-1}}{(q)_{\infty}^n}
\prod_{1\leq i<j\leq n} \theta(q^{j-i};q^{\kappa}) \\
&=\frac{(q^{\kappa};q^{\kappa})_{\infty}^{m-1}}{(q)_{\infty}^m} 
\prod_{1\leq i<j\leq m} \theta(q^{j-i};q^{\kappa}).
\end{align*}
\end{theorem}

\medskip
The remainder of this paper is organised as follows.
In the next section we recall some basic definitions and facts
from the theory of Hall--Littlewood polynomials and use this
to give an alternative, combinatorial 
representation for the left-hand side of \eqref{Eq_mixed}.
Then, in Sections~\ref{Sec_Pf} and \ref{Sec_Pf2}, we prove
Theorems~\ref{Thm_Main}--\ref{Thm_Main3} and
Theorem~\ref{Thm_Main4}, respectively,
and interpret each of the theorems
from the point of view of representation theory.

\section{The Hall--Littlewood polynomials}\label{Sec_HL}

Let $\la=(\la_1,\la_2,\dots)$ be a partition \cite{Andrews76}, i.e., 
$\la_1\geq\la_2\geq\cdots$ such that only finitely-many
$\la_i>0$. The positive $\la_i$ are called the parts
of $\la$ and the number of parts, denoted $l(\la)$, is the
length of $\la$. 
The size $\abs{\la}$ of $\la$ is the sum of its parts.
The diagram of $\la$ consists of $l(\la)$ left-aligned
rows of squares such that the $i$th row contains $\la_i$ squares.
For example, the diagram of $\nu=(6,4,4,2)$ 
of length $4$ and size $16$ is
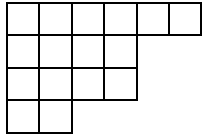
\[
\yng(6,4,4,2)
\]
The conjugate partition $\la'$ follows by transposing 
the diagram of $\la$.
For example, $\nu'=(4,4,3,3,1,1)$.
The nonnegative integers $m_i=m_i(\la)$, $i\geq 1$
give the multiplicities of parts of size $i$, so that
$\abs{\la}=\sum_i i m_i$.
It is easy to see that $m_i=\la'_i-\la'_{i+1}$.
We say that a partition is even if all its parts are even.
Note that $\la'$ is even if all multiplicities $m_i(\la)$
are even. The partition $\nu$ in our example is an even partition.
Given two partitions $\la,\mu$ we write $\mu\subseteq\la$ if the
diagram of $\mu$ is contained in the diagram of $\la$, or, equivalently,
if $\mu_i\leq \la_i$ for all $i$.
To conclude our discussion of partitions we define the 
generalised $q$-shifted factorial $b_{\la}(q)$ as
\begin{equation}\label{Eq_blambda}
b_{\la}(q)=\prod_{i\geq 1} (q)_{m_i}=\prod_{i\geq 1} (q)_{\la'_i-\la'_{i+1}}.
\end{equation}
Hence $b_{\nu}(q)=(q)_1^2(q)_2$.

For a fixed positive integer $n$, let $x=(x_1,\dots,x_n)$.
Given a partition $\la$ such that $l(\la)\leq n$,
write $x^{\la}$ for the monomial $x_1^{\la_1}\dots x_n^{\la_n}$
and define $v_{\la}(q)=\prod_{i=0}^n (q)_{m_i}/(1-q)^{m_i}$,
where $m_0:=n-l(\la)$.
The Hall--Littlewood polynomial $P_{\la}(x;q)$ 
is defined as the symmetric function \cite{Macdonald95}
\[
P_{\la}(x;q)=\frac{1}{v_{\la}(q)}
\sum_{w\in\Symm_n} 
w\bigg(x^{\la}\prod_{i<j}\frac{x_i-qx_j}{x_i-x_j}\bigg),
\]
where the symmetric group $\Symm_n$ acts on $x$ by permuting the $x_i$.
It follows from the definition that $P_{\la}(x;q)$ is
a homogeneous polynomial of degree $\abs{\la}$, a fact used repeatedly
in the rest of this paper.
$P_{\la}(x;q)$ is defined to be identically $0$ if $l(\la)>n$.
The Hall--Littlewood polynomials may be extended
in the usual way to symmetric functions in
countably-many variables, see \cite{Macdonald95}.
Below we only need this for $x$ a geometric progression.

For $x=(x_1,x_2,\dots)$ not necessarily finite, let $p_r$ be the $r$-th 
power sum symmetric function
\[
p_r(x)=x_1^r+x_2^r+\cdots,
\]
and $p_{\la}=\prod_{i\geq 1} p_{\la_i}$.
The power sums $\{p_{\la}(x_1,\dots,x_n)\}_{l(\la)\leq n}$ form a $\Rat$-basis of 
the ring of symmetric functions in $n$ variables.
If $\phi_q$ denotes the ring homomorphism $\phi_q(p_r)=p_r/(1-q^r)$,
then the modified Hall--Littlewood polynomials $P'_{\la}(x;q)$ 
are defined as the image of the $P_{\la}(x;q)$ under $\phi_q$: 
\[
P'_{\la}=\phi_q\big(P_{\la}\big).
\]

We also require the Hall--Littlewood polynomials $Q_{\la}$ and $Q'_{\la}$ 
defined by
\begin{equation}\label{Eq_QQp}
Q_{\la}(x;q)=b_{\la}(q) P_{\la}(x;q)\quad\text{and}\quad
Q'_{\la}(x;q)=b_{\la}(q) P'_{\la}(x;q).
\end{equation}
Clearly, $Q'_{\la}=\phi_q\big(Q_{\la}\big)$.

Up to the point where the $x$-variables are specialised, our proof 
of Theorems~\ref{Thm_Main}--\ref{Thm_Main3}
features the modified rather than the ordinary Hall--Littlewood polynomials.
Through specialisation we arrive at $P_{\la}$ evaluated
at a geometric progression thanks to
\begin{equation}\label{Eq_PpP}
P'_{\la}(1,q,\dots,q^{n-1};q^n)=P_{\la}(1,q,q^2,\dots;q^n),
\end{equation}
which readily follows from
\[
\phi_{q^n} \big( p_r(1,q,\dots,q^{n-1}) \big)
=\frac{1-q^{nr}}{1-q^r} \cdot \frac{1}{1-q^{nr}}
=p_r(1,q,q^2,\dots).
\]

From \cite{Kirillov00,WZ12} we may infer the following combinatorial
formula for the modified Hall--Littlewood polynomials:
\[
Q'_{\la}(x;q)=
\sum \prod_{i=1}^{\la_1}\prod_{a=1}^n
x_a^{\mu^{(a-1)}_i-\mu^{(a)}_i}
q^{\binom{\mu^{(a-1)}_i-\mu^{(a)}_i}{2}}
\qbin{\mu^{(a-1)}_i-\mu^{(a)}_{i+1}}{\mu^{(a-1)}_i-\mu^{(a)}_i}_q,
\]
where the sum is over partitions
$0=\mu^{(n)}\subseteq\cdots\subseteq\mu^{(1)}\subseteq\mu^{(0)}=\la'$
and
\[
\qbin{n}{m}_q=\begin{cases} \displaystyle \frac{(q)_{n}}{(q)_m(q)_{n-m}} &
\text{if $m\in\{0,1,\dots,n\}$} \\[3mm]
0 & \text{otherwise} \end{cases}
\]
is a $q$-binomial coefficient.
Therefore, by \eqref{Eq_blambda}--\eqref{Eq_PpP},
\begin{multline}\label{Eq_comb}
\sum_{\substack{\la \\[1pt] \la_1\leq m}} 
q^{(\sigma+1)\abs{\la}} P_{2\la}(1,q,q^2,\dots;q^n)\\
=\sum \prod_{i=1}^{2m}\Bigg\{
\frac{q^{\frac{1}{2}(\sigma+1)\mu^{(0)}_i}}
{(q^n;q^n)_{\mu^{(0)}_i-\mu^{(0)}_{i+1}}}
\prod_{a=1}^n
q^{\mu_i^{(a)}+n\binom{\mu^{(a-1)}_i-\mu^{(a)}_i}{2}}
\qbin{\mu^{(a-1)}_i-\mu^{(a)}_{i+1}}{\mu^{(a-1)}_i-\mu^{(a)}_i}_{q^n}
\Bigg\},
\end{multline}
where the sum on the right is over partitions
$0=\mu^{(n)}\subseteq\cdots\subseteq\mu^{(1)}\subseteq\mu^{(0)}$
such that $(\mu^{(0)})'$ is even and $l(\mu^{(0)})\leq 2m$.
This may be used to express the sum sides of 
\eqref{Eq_RR-A2n2}--\eqref{Eq_mixed} combinatorially.
To see that \eqref{Eq_comb} indeed generalises the 
sums in \eqref{Eq_RR} and \eqref{Eq_AG}, we note that the above
simplifies for $n=1$ to
\[
\sum_{\substack{\la \\[1pt] \la_1\leq m}} 
q^{(\sigma+1)\abs{\la}} P_{2\la}(1,q,q^2,\dots;q)
=\sum
\prod_{i=1}^{2m} \frac{q^{\frac{1}{2}\mu_i(\mu_i+\sigma)}}{(q)_{\mu_i-\mu_{i+1}}}
\]
summed on the right over partitions $\mu$ of length at most $2m$ whose conjugates are even.
Such partitions are characterised by the restriction 
$\mu_{2i}=\mu_{2i-1}=:r_i$ so that we get
\[
\sum_{\substack{\la \\[1pt] \la_1\leq m}} 
q^{(\sigma+1)\abs{\la}} P_{2\la}(1,q,q^2,\dots;q)
=\sum_{r_1\geq\cdots\geq r_m\geq 0} \prod_{i=1}^m
\frac{q^{r_i(r_i+\sigma)}}{(q)_{r_i-r_{i+1}}}
\]
in accordance with \eqref{Eq_AG}.
If instead we consider $m=1$ and replace
$\mu^{(j)}$ by $(r_j,s_j)$ for $j\geq 0$,
we find
\begin{align*}
&\sum_{r=0}^{\infty} q^{(\sigma+1)r} P_{(2^r)}(1,q,q^2,\dots;q^n)\\
&\quad=\sum
\frac{q^{(\sigma+1)r_0}}{(q^n;q^n)_{r_0}}
\prod_{j=1}^n
q^{r_j+s_j+n\binom{r_{j-1}-r_j}{2}
+n\binom{s_{j-1}-s_j}{2}}
\qbin{r_{j-1}-s_j}{r_{j-1}-r_j}_{q^n} 
\qbin{s_{j-1}}{s_j}_{q^n} \\
&\quad=\frac{(q^{n+4};q^{n+4})_{\infty}}{(q)_{\infty}}\, 
\theta\big(q^{2-\sigma};q^{n+4}\big),
\end{align*}
where the second sum is over $r_0,s_0,\dots,r_{n-1},s_{n-1}$ such that
$r_0=s_0$, and $r_n=s_n:=0$.

We conclude this section with a remark about Theorem~\ref{Thm_Main4}.
Due to the occurrence of the limit, the left-hand side does not take the
form of the usual sum-side of a Rogers--Ramanujan-type identity.
For special cases it is, however, possible to eliminate the limit.
For example, for partitions of the form $(2^r)$ we found that
\begin{multline}\label{Eq_Q2r}
P_{(2^r)}(1,q,q^2,\dots;q^{2n+\delta}) \\
=\sum_{r\geq r_1\geq\dots\geq r_n\geq 0}
\frac{q^{r^2-r+r_1^2+\cdots+r_n^2+r_1+\cdots+r_n}}
{(q)_{r-r_1}(q)_{r_1-r_2}
\cdots(q)_{r_{n-1}-r_n}(q^{2-\delta};q^{2-\delta})_{r_n}}
\end{multline}
for $\delta=0,1$.
This turns the $m=2$ case of Theorem~\ref{Thm_Main4} into
\begin{multline*}
\sum_{r_1\geq\dots\geq r_n\geq 0}
\frac{q^{r_1^2+\cdots+r_n^2+r_1+\cdots+r_n}}
{(q)_{r_1-r_2}\cdots(q)_{r_{n-1}-r_n}(q^{2-\delta};q^{2-\delta})_{r_n}} \\
=\frac{(q^{2n+2+\delta};q^{2n+2+\delta})_{\infty}}
{(q)_{\infty}}\, \theta(q;q^{2n+2+\delta}).
\end{multline*}
For $\delta=1$ this is the $i=1$ case of the
Andrews--Gordon identity \eqref{Eq_AG} (with $m$ replaced by $n$). 
For $\delta=0$ it corresponds 
to an identity due to Bressoud \cite{Bressoud80}.
We do not know how to generalise \eqref{Eq_Q2r}
to arbitrary rectangular shapes.

\section{Proof of Theorems~\ref{Thm_Main}--\ref{Thm_Main3}}\label{Sec_Pf}

\subsection{The Watson--Andrews approach}

In 1929 Watson proved the Rog\-ers--Ra\-man\-u\-jan identities \eqref{Eq_RR}
by first proving a new basic hypergeometric series transformation between
a terminating balanced $_4\phi_3$ series and a terminating 
very-well-poised $_8\phi_7$ series \cite{Watson29}
\begin{multline}\label{Eq_Watson}
\frac{(aq,aq/bc)_N}{(aq/b,aq/c)_N}
\sum_{r=0}^N \frac{(b,c,aq/de,q^{-N})_r}
{(q,aq/d,aq/e,bcq^{-N}/a)_r} \, q^r \\
=\sum_{r=0}^N \frac{1-a q^{2r}}{1-a}\, 
\frac{(a,b,c,d,e,q^{-N})_r}{(q,aq/b,aq/c,aq/d,aq/e)_r}
\bigg(\frac{a^2q^{N+2}}{bcde}\bigg)^r .
\end{multline} 
Here $a,b,c,d,e$ are indeterminates, $N$ is a nonnegative integer
and
\[
(a_1,\dots,a_m)_k=(a_1,\dots,q_m;q)=(a_1;q)_k\cdots (a_m;q)_k.
\]
By letting $b,c,d,e$ tend to infinity and taking the
nonterminating limit $N\to\infty$,
Watson arrived at what is known as the Rogers--Selberg 
identity \cite{Rogers19,Selberg36}\footnote{Here and elsewhere in the paper
we ignore questions of convergence. From an analytic point of view 
the transition from \eqref{Eq_Watson} to \eqref{Eq_RS} requires
the use of the dominated convergence theorem, imposing the 
restriction $\abs{q}<1$ on the Rogers--Selberg identity.
We however choose to view this identity as an identity between 
formal power series in $q$,
in line with the combinatorial and representation-theoretic 
interpretations of Rogers--Ramanujan-type identities.} 
\begin{equation}\label{Eq_RS}
\sum_{r=0}^{\infty} \frac{a^r q^{r^2}}{(q)_r} \\
=\frac{1}{(aq)_{\infty}} \sum_{r=0}^{\infty}
\frac{1-a q^{2r}}{1-a}\,\frac{(a)_r}{(q)_r}\,
(-1)^r a^{2r} q^{5\binom{r}{2}+2r}.
\end{equation}
For $a=1$ or $a=q$ the sum on the right can be expressed in product-form 
by the Jacobi triple-product identity
\[
\sum_{r=-\infty}^{\infty} (-1)^r x^r q^{\binom{r}{2}}
=(q)_{\infty}\,\theta(x;q),
\]
resulting in \eqref{Eq_RR}.

Almost 50 years after Watson's work, Andrews showed that the 
Andrews--Gordon identities \eqref{Eq_AG} for $i=1$ and $i=m+1$ follow in 
much the same manner from a multiple series generalisation of 
\eqref{Eq_Watson} in which the $_8\phi_7$ series on the right is replaced 
by a terminating very-well-poised $_{2m+6}\phi_{2m+5}$ series depending 
on $2m+2$ parameters instead of $b,c,d,e$ \cite{Andrews75}.
Again the key steps are to let all these parameters
tend to infinity, to take the nonterminating limit 
and express the $a=1$ or $a=q$ instances of the resulting
sum as a product by the Jacobi triple-product identity.

Recently, in joint work with Bartlett, we obtained an analogue of
Andrews' multiple series transformation for the $\C_n$ root system 
\cite[Theorem 4.2]{BW13}. 
Apart from the variables $(x_1,\dots,x_n)$---which play the role of
$a$ in \eqref{Eq_Watson}, and are related to the underlying root 
system---the $\C_n$ Andrews transformation again contains $2m+2$
parameters.
Unfortunately, simply following the Andrews--Watson procedure
is no longer sufficient. In \cite{Milne94} Milne already 
obtained the $\C_n$ analogue of the Rogers--Selberg identity 
\eqref{Eq_RS} (the $m=1$ case of \eqref{Eq_Cn-RS-m} below) and 
considered specialisations along the lines of Andrews and Watson. 
Only for $\C_2$ did this result in a Rogers--Ramanujan-type identity:
the modulus $6$ case of \eqref{Eq_RR-Dn} mentioned previously.

The initial two steps towards proof of \eqref{Eq_RR-A2n2}--\eqref{Eq_mixed},
however, are the same as those of Watson and Andrews:
we let all $2m+2$ parameters in 
the $\C_n$ Andrews transformation tend to infinity and take the 
nonterminating limit. Then, as shown in \cite{BW13}, the right-hand side 
can be expressed in terms of modified Hall--Littlewood polynomials, 
resulting in the level-$m$ $\C_n$ Rogers--Selberg identity
\begin{equation}\label{Eq_Cn-RS-m}
\sum_{\substack{\la \\[1pt] \la_1\leq m}}
q^{\abs{\la}} P'_{2\la}(x;q)=L^{(0)}(x;q)
\end{equation}
for 
\begin{multline*}
L^{(0)}_m(x;q) :=
\sum_{r\in\NN^n}\frac{\Delta_{\C}(x q^r)}{\Delta_{\C}(x)}\,
\prod_{i=1}^n x_i^{2(m+1)r_i} q^{(m+1)r_i^2+n\binom{r_i}{2}} \\
\times \prod_{i,j=1}^n \Big({-}\frac{x_i}{x_j}\Big)^{r_i} 
\frac{(x_i x_j)_{r_i}}{(q x_i/x_j)_{r_i}}.
\end{multline*}
Here
\[
\Delta_{\C}(x):=\prod_{i=1}^n (1-x_i^2)\prod_{1\leq i<j\leq n}
(x_i-x_j)(x_ix_j-1)
\]
is the $\C_n$ Vandermonde product and $f(xq^r)$ is short\-hand for 
$f(x_1q^{r_1},\dots,x_nq^{r_n})$.
As mentioned previously, \eqref{Eq_Cn-RS-m} for $m=1$ is Milne's 
$\C_n$ Rogers--Selberg formula \cite[Corollary 2.21]{Milne94}.

Comparing the left-hand side of \eqref{Eq_Cn-RS-m} with that of 
\eqref{Eq_RR-A2n2}--\eqref{Eq_RR-Dn} it follows that we should 
make the simultaneous substitutions
\begin{equation}\label{Eq_spec}
q\mapsto q^n,\qquad x_i\mapsto q^{(n+\sigma+1)/2-i}\;\; (1\leq i\leq n).
\end{equation}
Then, by the homogeneity and symmetry of the (modified) Hall--Littlewood 
polynomials and \eqref{Eq_PpP},
\[
\sum_{\substack{\la \\[1pt] \la_1\leq m}}
q^{\abs{\la}} P'_{2\la}(x;q)  \longmapsto
\sum_{\substack{\la \\[1pt] \la_1\leq m}}
q^{(\sigma+1)\abs{\la}} P_{2\la}(1,q,q^2,\dots;q^n).
\]
The problem we face is that making the substitution \eqref{Eq_spec} 
on the right-hand side of \eqref{Eq_Cn-RS-m} 
and then writing the resulting $q$-series
in product form is very difficult.
To get around this problem, we take a rather different route and 
(up to a small constant) first double the rank of the 
underlying $\C_n$ root system
and then take a limit in which products of pairs of $x$-variables tend 
to one. To do so we require another result from \cite{BW13}.

First we need to extend our earlier definition of the $q$-shifted
factorial to $(a)_k=(a)_{\infty}/(aq^k)_{\infty}$.
Importantly, $1/(q)_k=0$ for $k$ a negative integer. Then,
for $x=(x_1,\dots,x_n)$, $p$ an integer such that $0\leq p\leq n$ and 
$r\in\Z^n$,
\begin{multline}\label{Eq_Def-Lp}
L^{(p)}_m(x;q):=
\sum_{r\in\Z^n} \frac{\Delta_{\C}(x q^r)}{\Delta_{\C}(x)}\,
\prod_{i=1}^n x_i^{2(m+p+1)r_i}
q^{(m+1)r_i^2+(n+p)\binom{r_i}{2}} \\
\times
\prod_{i=1}^n \prod_{j=p+1}^n
\Big({-}\frac{x_i}{x_j}\Big)^{r_i} 
\frac{(x_i x_j)_{r_i}}{(qx_i/x_j)_{r_i}}.
\end{multline}
Note that the summand of $L^{(p)}_m(x;q)$ vanishes 
if one of $r_{p+1},\dots,r_n<0$.
\begin{lemma}[\!\!{\cite[Lemma~A.1]{BW13}}]
For $1\leq p\leq n-1$,
\begin{equation}\label{Eq_key}
\lim_{x_{p+1}\to x_p^{-1}} L^{(p-1)}_m(x;q)=
L^{(p)}_m(x_1,\dots,x_{p-1},x_{p+1},\dots,x_n;q).
\end{equation}
\end{lemma}
This will be the key to the proof of all four generalised 
Rogers--Ramanujan identities, although the level of difficulty 
varies considerably from case to case.
We begin with the simplest proof, that of \eqref{Eq_RR-Cn}.

\subsection{Proof of the \eqref{Eq_RR-Cn}} 

By iterating \eqref{Eq_key} we obtain
\[
\lim_{y_1\to x_1^{-1}}\dots \lim_{y_n\to x_n^{-1}}
L^{(0)}_m(x_1,y_1,\dots,x_n,y_n)=L^{(n)}_m(x_1,\dots,x_n).
\]
Hence, after replacing $x\mapsto(x_1,y_1,\dots,x_n,y_n)$ in 
\eqref{Eq_Cn-RS-m} (which corresponds to the doubling of the rank
mentioned previously) and taking the $y_i\to x_i^{-1}$ limit 
for $1\leq i\leq n$, we find
\begin{multline}\label{Eq_CnmLa0}
\sum_{\substack{\la \\[1pt] \la_1\leq m}}
q^{\abs{\la}} P'_{2\la}(x^{\pm};q)=
\frac{1}{(q)_{\infty}^n\prod_{i=1}^n \theta(x_i^2;q)
\prod_{1\leq i<j\leq n} \theta(x_i/x_j,x_ix_j;q)} \\ \times
\sum_{r\in\Z^n} \Delta_{\C}(x q^r)
\prod_{i=1}^n x_i^{\kappa r_i-i+1} q^{\frac{1}{2}\kappa r_i^2-nr_i},
\end{multline}
where $\kappa=2m+2n+2$ and $f(x^{\pm})=f(x_1,x_1^{-1},\dots,x_n,x_n^{-1})$.
Next we make the simultaneous substitutions
\begin{equation}\label{Eq_subs}
q\mapsto q^{2n},\qquad x_i\mapsto q^{n-i+1/2}=:\hat{x}_i\;\; (1\leq i\leq n),
\end{equation}
which corresponds to \eqref{Eq_spec} with $(n,\sigma)\mapsto (2n,0)$.
By 
\[
(q^{2n};q^{2n})_{\infty}^n\prod_{i=1}^n \theta(q^{2n-2i+1};q^{2n})
\prod_{1\leq i<j\leq n} \theta(q^{j-i},q^{2n-i-j+1};q^{2n})=
\frac{(q)_{\infty}^{n+1}}{(q^2;q^2)_{\infty}},
\]
and
\begin{align*}
q^{2n\abs{\la}} & P'_{2\la}(q^{n-1/2},q^{1/2-n},\dots,q^{1/2},q^{-1/2};q^{2n}) && \\
&=q^{2n\abs{\la}} P'_{2\la}(q^{1/2-n},q^{3/2-n},\dots,q^{n-1/2};q^{2n}) 
&& \text{by symmetry} \\
&=q^{\abs{\la}} P'_{2\la}(1,q,\dots,q^{2n-1};q^{2n}) && \text{by homogeneity} \\ 
&=q^{\abs{\la}} P_{2\la}(1,q,q^2,\dots;q^{2n}) && \text{by \eqref{Eq_PpP}},
\end{align*}
this yields
\begin{equation}\label{Eq_Mdef}
\sum_{\substack{\la \\[1pt] \la_1\leq m}}
q^{\abs{\la}} P_{2\la}\big(1,q,q^2,\dots;q^{2n}\big) =
\frac{(q^2;q^2)_{\infty}}{(q)_{\infty}^{n+1}}\,
\M,
\end{equation}
where 
\[
\M:=
\sum_{r\in\Z^n} \Delta_{\C}(\hat{x} q^{2nr})\,
\prod_{i=1}^n \hat{x}_i^{\kappa r_i-i+1} q^{nr_i^2-2n^2 r_i}.
\]

What remains is to express $\M$ in product form. 
As a first step we use the $\C_n$ Weyl denominator
formula \cite[Lemma 2]{Krattenthaler99}
\begin{equation}\label{Eq_detC}
\Delta_{\C}(x)=\det_{1\leq i,j\leq n} \big(x_i^{j-1}-x_i^{2n-j+1}\big)
\end{equation}
as well as multilinearity, to write $\M$ as
\begin{equation}\label{Eq_begin}
\M=\det_{1\leq i,j\leq n} 
\bigg(\sum_{r\in\Z} \hat{x}_i^{\kappa r-i+1}
q^{n\kappa r^2-2n^2r}
\Big( (\hat{x}_i q^{2nr})^{j-1}-(\hat{x}_iq^{2nr})^{2n-j+1}\Big)\bigg).
\end{equation}
We now replace $(i,j)\mapsto (n-j+1,n-i+1)$ and, viewing the
resulting determinant as being of the form 
$\det\big(\sum_r u_{ij;r}-\sum_r v_{ij;r}\big)$, we change the summation index
$r\mapsto -r-1$ in the sum over $v_{ij;r}$. Then
\begin{equation}\label{Eq_end}
\M=\det_{1\leq i,j\leq n}\bigg(
q^{a_{ij}} \sum_{r\in\Z} y_i^{2nr-i+1}
q^{2n\kappa\binom{r}{2}+\frac{1}{2}\kappa r} 
\Big((y_iq^{\kappa r})^{j-1}-(y_i q^{\kappa r})^{2n-j}\Big)\bigg),
\end{equation}
where $y_i=q^{\kappa/2-i}$ and 
$a_{ij}=j^2-i^2+(i-j)(\kappa+1)/2$.
Since the factor $q^{a_{ij}}$ does not contribute to the
determinant, we can apply the $\B_n$ Weyl denominator formula 
\cite{Krattenthaler99}
\begin{equation}\label{Eq_Bn-VdM}
\det_{1\leq i,j\leq n} \big(x_i^{j-1}-x_i^{2n-j}\big)=
\prod_{i=1}^n (1-x_i)
\prod_{1\leq i<j\leq n} (x_i-x_j)(x_ix_j-1)=:\Delta_{\B}(x)
\end{equation}
to obtain
\[
\M=\sum_{r\in\Z^n} \Delta_{\B}(y q^{\kappa r})
\prod_{i=1}^n y_i^{2nr_i-i+1} 
q^{2n\kappa\binom{r_i}{2}+\frac{1}{2}\kappa r_i}.
\]
By the $\D_{n+1}^{(2)}$ Macdonald identity \cite{Macdonald72}
\begin{multline*}
\sum_{r\in\Z^n} \Delta_{\B}(xq^r)
\prod_{i=1}^n x_i^{2nr_i-i+1} 
q^{2n\binom{r_i}{2}+\frac{1}{2}r_i} \\
=(q^{1/2};q^{1/2})_{\infty}(q)_{\infty}^{n-1}
\prod_{i=1}^n \theta(x_i;q^{1/2})_{\infty}
\prod_{1\leq i<j\leq n} \theta(x_i/x_j,x_ix_j;q)
\end{multline*}
with $(q,x)\mapsto (q^{\kappa},y)$ this results in
\begin{equation}\label{Eq_Mcn}
\M=(q^{\kappa/2};q^{\kappa/2})_{\infty}
(q^{\kappa};q^{\kappa})_{\infty}^{n-1}
 \prod_{i=1}^n  \theta\big(q^i;q^{\kappa/2}\big)
\prod_{1\leq i<j\leq n} \theta\big(q^{j-i},q^{i+j};q^{\kappa}\big),
\end{equation}
where we have also used the simple symmetry 
$\theta(q^{a-b};q^a)=\theta(q^b;q^a)$.
Substituting \eqref{Eq_Mcn} into \eqref{Eq_Mdef} proves the first equality 
of \eqref{Eq_RR-Cn}.

To show that the second equality holds is a straightforward
exercise in manipulating infinite products, and we omit the details.

\medskip

There is a somewhat different approach to \eqref{Eq_RR-Cn} based on
the representation theory of the affine Kac--Moody algebra 
$\C_n^{(1)}$ \cite{Kac90}. 
Let $I=\{0,1,\dots,n\}$, and
$\alpha_i$, $\alpha^{\vee}_i$ and $\La_i$ for $i\in I$
the simple roots, simple coroots and fundamental weights
of $\C_n^{(1)}$.
Let $\ip{\cdot}{\cdot}$ denote the usual pairing between the Cartan
subalgebra $\hfrak$ and its dual $\hfrak^{\ast}$, 
so that $\ip{\La_i}{\alpha_j^{\vee}}=\delta_{ij}$.
Finally, let $V(\La)$ be the integrable highest-weight module of 
$\C_n^{(1)}$ of highest weight $\La$ with character $\ch V(\La)$.  

The homomorphism
\begin{equation}\label{Eq_PS}
F_{\mathds{1}}:~\mathbb{C}[[\eup^{-\alpha_0},\dots,\eup^{-\alpha_n}]]
\to \mathbb{C}[[q]],\qquad 
F_{\mathds{1}}(\eup^{-\alpha_i})=q\quad\text{for all $i\in I$}
\end{equation}
is known as principal specialisation.
In \cite{Kac74} Kac showed that the principally specialised characters
admit a product form.
Let $\rho$ be the Weyl vector (that is $\ip{\rho}{\alpha_i^{\vee}}=1$
for $i\in I$) and $\mult(\alpha)$ the multiplicity of $\alpha$.
Then Kac's formula is given by
\begin{equation}\label{Eq_Kac}
F_{\mathbbm{1}}\big(\eup^{-\La} \ch V(\La)\big)=\prod_{\alpha\in\Delta_+^{\vee}}
\bigg(\frac{1-q^{\ip{\La+\rho}{\alpha}}}
{1-q^{\ip{\rho}{\alpha}}}\bigg)^{\mult(\alpha)},
\end{equation}
where $\Delta_+^{\vee}$ is the set of positive coroots.
This result, which is valid for all types $\mathrm{X}_N^{(r)}$, 
can be rewritten in terms of theta functions.
Assuming $\C_n^{(1)}$ and setting
\[
\La=(\la_0-\la_1)\La_0+(\la_1-\la_2)\La_1+\cdots+(\la_{n-1}-\la_n)\La_{n-1}+\la_n\La_n,
\] 
for $\la=(\la_0,\la_1,\dots,\la_n)$ a partition, this rewriting takes the form
\begin{multline}\label{Eq_Kac2}
F_{\mathbbm{1}}\big(
\eup^{-\La} \ch V(\La)\big)
=\frac{(q^2;q^2)_{\infty}(q^{\kappa/2};q^{\kappa/2})_{\infty}
(q^{\kappa};q^{\kappa})_{\infty}^{n-1}} 
{(q;q)_{\infty}^{n+1}}  \\ \times
\prod_{i=1}^n  \theta\big(q^{\la_i+n-i+1};q^{\kappa/2}\big) 
\prod_{1\leq i<j\leq n} 
\theta\big(q^{\la_i-\la_j-i+j},q^{\la_i+\la_j+2n+2-i-j};q^{\kappa}\big),
\end{multline}
where $\kappa=2n+2\la_0+2$.

The earlier product form now arises by recognising (see e.g.,
\cite[Lemma 2.1]{BW13}) the right-hand side of \eqref{Eq_CnmLa0} as
\[
\eup^{-m\La_0} \ch V(m\La_0)
\]
upon the identification
\[
q=\eup^{-\alpha_0-2\alpha_1-\cdots-2\alpha_{n-1}-\alpha_n}
\quad\text{and}\quad
x_i=\eup^{-\alpha_i-\cdots-\alpha_{n-1}-\alpha_n/2}\;\;
(1\leq i\leq n).
\]
Since \eqref{Eq_subs} corresponds exactly
to the principal specialisation \eqref{Eq_PS}
it follows from \eqref{Eq_Kac2} with $\la=(m,0^n)$ that
\begin{multline*}
F_{\mathbbm{1}}\big(\eup^{-m\La_0} \ch V(m\La_0)\big) 
=
\frac{(q^2;q^2)_{\infty}(q^{\kappa/2};q^{\kappa/2})_{\infty}
(q^{\kappa};q^{\kappa})_{\infty}^{n-1}} 
{(q;q)_{\infty}^{n+1}}   \\ \times
\prod_{i=1}^n  \theta\big(q^{n-i+1};q^{\kappa/2}\big) 
\prod_{1\leq i<j\leq n} 
\theta\big(q^{j-i},q^{i+j};q^{\kappa}\big).
\end{multline*}
We should remark that this representation-theoretic approach 
is not essentially different from our earlier $q$-series 
proof. Indeed, the principal specialisation formula \eqref{Eq_Kac2} 
itself is an immediate consequence of the $\D^{(2)}_{n+1}$ Macdonald 
identity, and if, instead of the right-hand side of
\eqref{Eq_CnmLa0}, we consider the more general
\begin{multline*}
\eup^{-\La} \ch V(\La)=
\frac{1}{(q)_{\infty}^n\prod_{i=1}^n \theta(x_i^2;q)
\prod_{1\leq i<j\leq n} \theta(x_i/x_j,x_ix_j;q)} \\ \times
\sum_{r\in\Z^n} 
\det_{1\leq i,j\leq n}
\Big( (x_i q^{r_i})^{j-\la_i-1}-(x_i q^{r_i})^{2n-j+\la_i+1}\Big)
\prod_{i=1}^n x_i^{\kappa r_i+\la_i-i+1} q^{\frac{1}{2}\kappa r_i^2-nr_i}
\end{multline*}
for $\kappa=2n+2\la_0+2$, then all of the steps carried out between 
\eqref{Eq_CnmLa0} and \eqref{Eq_Mcn} carry over to this more
general setting. The only notable changes are that
\eqref{Eq_begin} generalises to
\begin{multline*}
\M=\det_{1\leq i,j\leq n} 
\bigg(\sum_{r\in\Z} \hat{x}_i^{\kappa r+\la_i-i+1}
q^{n\kappa r^2-2n^2r} \\ \times
\Big( (\hat{x}_i q^{2nr})^{j-\la_i-1}-(\hat{x}_iq^{2nr})^{2n-j+\la_i+1}\Big)
\bigg).
\end{multline*}
and that in \eqref{Eq_end} we have to redefine
$y_i$ and $a_{ij}$ as $q^{\kappa/2-\la_{n-i+1}-i}$ and 
$j^2-i^2+(i-j)(\kappa+1)/2+(j-1/2)\la_{n-j+1}-(i-1/2)\la_{n-i+1}$.

\subsection{Proof of the \eqref{Eq_RR-A2n2a}}

Again we iterate \eqref{Eq_key}, but this time the variable $x_n$, remains
unpaired: 
\[
\lim_{y_1\to x_1^{-1}}\dots \lim_{y_{n-1}\to x_{n-1}^{-1}}
L^{(0)}_m(x_1,y_1,\dots,x_{n-1},y_{n-1},x_n)=L^{(n-1)}_m(x_1,\dots,x_n).
\]
Therefore, if we replace $x\mapsto(x_1,y_1,\dots,x_{n-1},y_{n-1},x_n)$ 
in \eqref{Eq_Cn-RS-m} (changing the rank from $n$ to $2n-1$) 
and take the $y_i\to x_i^{-1}$ limit for $1\leq i\leq n-1$, we obtain
\begin{align}\label{Eq_interm}
\sum_{\substack{\la \\[1pt] \la_1\leq m}} &
q^{\abs{\la}} P'_{2\la}\big(x_1^{\pm},\dots,x_{n-1}^{\pm},x_n;q\big) \\[-1mm]
&=\frac{1}{(q)_{\infty}^{n-1}(qx_n^2)_{\infty} 
\prod_{i=1}^{n-1}  (qx_i^{\pm}x_n,qx_i^{\pm 2})_{\infty}
\prod_{1\leq i<j\leq n-1} (qx_i^{\pm} x_j^{\pm})_{\infty}} \notag \\[1mm] 
&\qquad \times
\sum_{r\in\Z^n} \frac{\Delta_{\C}(x q^r)}{\Delta_{\C}(x)}\,
\prod_{i=1}^n \bigg({-}\frac{x_i^{\kappa}}{x_n}\bigg)^{r_i} 
q^{\frac{1}{2}\kappa r_i^2-\frac{1}{2}(2n-1)r_i} 
\frac{(x_i x_n)_{r_i}}{(qx_i/x_n)_{r_i}}, \notag
\end{align}
where $\kappa=2m+2n+1$, 
$(ax_i^{\pm})_{\infty}:=(ax_i)_{\infty}(ax_i^{-1})_{\infty}$ and
\[
(ax_i^{\pm}x_j^{\pm})_{\infty}:=
(ax_ix_j)_{\infty}(ax_i^{-1}x_j)_{\infty}
(ax_ix_j^{-1})_{\infty}(ax_i^{-1}x_j^{-1})_{\infty}.
\]
Recalling the comment immediately after \eqref{Eq_Def-Lp},
the summand of \eqref{Eq_interm} vanishes unless $r_n\geq 0$.

Let $\hat{x}:=(-x_1,\dots,-x_{n-1},-1)$ and
\begin{equation}\label{Eq_phi}
\phi_r=\begin{cases} 1 & \text{if $r=0$} \\ 2 & \text{if $r=1,2,\dots$.}
\end{cases}
\end{equation}
Letting $x_n$ tend to $1$ in \eqref{Eq_interm} using
\[
\lim_{x_n\to 1} 
\frac{\Delta_{\C}(x q^r)}{\Delta_{\C}(x)}\,
\prod_{i=1}^n \frac{(x_ix_n)_{r_i}}{(qx_i/x_n)_{r_i}}
=\phi_{r_n} 
\frac{\Delta_{\B}(\hat{x} q^r)}{\Delta_{\B}(\hat{x})},
\]
we find
\begin{align*}
\sum_{\substack{\la \\[1pt] \la_1\leq m}} &
q^{\abs{\la}} P'_{2\la}\big(x_1^{\pm},\dots,x_{n-1}^{\pm},1;q\big) \\[-1mm] 
&=\frac{1}{(q)_{\infty}^n 
\prod_{i=1}^{n-1}  (qx_i^{\pm},qx_i^{\pm 2})_{\infty}
\prod_{1\leq i<j\leq n-1} (qx_i^{\pm} x_j^{\pm})_{\infty}} \\[1mm] 
& \qquad \times
\sum_{r_1,\dots,r_{n-1}=-\infty}^{\infty} \sum_{r_n=0}^{\infty}
\phi_{r_n} \frac{\Delta_{\B}(\hat{x} q^r)}{\Delta_{\B}(\hat{x})}\,
\prod_{i=1}^n \hat{x}_i^{\kappa r_i} 
q^{\frac{1}{2}\kappa r_i^2-\frac{1}{2}(2n-1)r_i}.
\end{align*}
It is easily checked that the summand on the right (without the factor
$\phi_{r_n}$) is invariant under the variable change
$r_n\mapsto -r_n$. Using the elementary relations
\begin{equation}\label{Eq_simp}
\theta(-1;q)=2(-q)_{\infty}^2,\quad
(-q)_{\infty}(q;q^2)_{\infty}=1,\quad
\theta(z,-z;q)\theta(qz^2;q^2)=\theta(z^2),
\end{equation}
we can thus simplify the above to
\begin{align}\label{Eq_a2n2}
\sum_{\substack{\la \\[1pt] \la_1\leq m}} &
q^{\abs{\la}} P'_{2\la}\big(x_1^{\pm},\dots,x_{n-1}^{\pm},1;q\big) \\[-1mm] 
&=\frac{1}{(q)_{\infty}^n 
\prod_{i=1}^n \theta(\hat{x}_i;q)\theta(q\hat{x}_i^2;q^2)
\prod_{1\leq i<j\leq n} 
\theta(\hat{x}_i/\hat{x}_j,\hat{x}_i\hat{x}_j;q)} \notag \\[1mm] 
& \qquad \times
\sum_{r\in\Z^n} \Delta_{\B}(\hat{x} q^r)\,
\prod_{i=1}^n \hat{x}_i^{\kappa r_i-i+1} 
q^{\frac{1}{2}\kappa r_i^2-\frac{1}{2}(2n-1)r_i}. \notag 
\end{align}

The remainder of the proof is similar to that of \eqref{Eq_RR-Cn}.
We make the simultaneous substitutions
\begin{equation}\label{Eq_A2n2-spec}
q\mapsto q^{2n-1},\qquad x_i\mapsto q^{n-i}\;\; (1\leq i\leq n),
\end{equation}
so that from here on $\hat{x}_i:=-q^{n-i}$. 
By 
\begin{multline*}
(q^{2n-1};q^{2n-1})_{\infty}^n\prod_{i=1}^n 
\theta(-q^{n-i};q^{2n-1})\theta(q^{2n-2i+1};q^{4n-2}) \\ \times
\prod_{1\leq i<j\leq n} \theta(q^{j-i},q^{2n-i-j};q^{2n-1})=
2(q)_{\infty}^n
\end{multline*}
and \eqref{Eq_PpP}, this results in
\[
\sum_{\substack{\la \\[1pt] \la_1\leq m}}
q^{\abs{\la}} P_{2\la}\big(1,q,q^2,\dots;q^{2n-1}\big) =
\frac{\M}{2(q)_{\infty}^n}, 
\] 
for
\[
\M:=\sum_{r\in\Z^n} \Delta_{\B}\big(\hat{x} q^{(2n-1)r}\big)\,
\prod_{i=1}^n \hat{x}_i^{\kappa r_i-i+1}
q^{\frac{1}{2}(2n-1)\kappa r_i^2-\frac{1}{2}(2n-1)^2r_i}.
\]
By \eqref{Eq_Bn-VdM} and multilinearity $\M$ can be rewritten in the form
\begin{multline*}
\M=\det_{1\leq i,j\leq n} 
\bigg(\sum_{r\in\Z} \hat{x}_i^{\kappa r-i+1}
q^{\frac{1}{2}(2n-1)\kappa r^2-\frac{1}{2}(2n-1)^2r} \\ \times
\Big( \big(\hat{x}_i q^{(2n-1)r}\big)^{j-1}-
\big(\hat{x}_iq^{(2n-1)r}\big)^{2n-j}\Big)\bigg).
\end{multline*}
Following the same steps that led from \eqref{Eq_begin} to \eqref{Eq_end} 
yields
\begin{multline}\label{Eq_same}
\M=\det_{1\leq i,j\leq n}\bigg(
(-1)^{i-j} q^{b_{ij}} \sum_{r\in\Z} (-1)^r 
y_i^{(2n-1)r-i+1} q^{(2n-1)\kappa\binom{r}{2}}  \\ \times
\Big((y_iq^{\kappa r})^{j-1}-(y_i q^{\kappa r})^{2n-j}\Big)\bigg),
\end{multline}
where 
\begin{equation}\label{Eq_yb}
y_i=q^{\frac{1}{2}(\kappa+1)-i} \quad\text{and}\quad 
b_{ij}:=j^2-i^2+\frac{1}{2}(i-j)(\kappa+3).
\end{equation} 
Again the factor $(-1)^{i-j} q^{b_{ij}}$ does not contribute
and application of \eqref{Eq_Bn-VdM} gives
\[
\M=\sum_{r\in\Z^n} 
\Delta_{\B}(y_iq^{\kappa r})
\prod_{i=1}^n (-1)^{r_i}  
y_i^{(2n-1)r_i-i+1} q^{(2n-1)\kappa\binom{r_i}{2}}.  
\]
To complete the proof we apply the following variant of the 
$\B_n^{(1)}$ Macdonald identity\footnote{The actual
$\B_n^{(1)}$ Macdonald identity has the restriction $\abs{r}\equiv 0\pmod{2}$
in the sum over $r\in\Z^n$, which eliminates the factor $2$
on the right. To prove the form used here it suffices to take the
$a_1,\dots,a_{2n-1}\to 0$ and $a_{2n}\to -1$ limit in
Gustafson's multiple $_6\psi_6$ summation for the affine root system 
$\A_{2n-1}^{(2)}$, see \cite{Gustafson89}.}
\begin{multline}\label{Eq_Gustafson}
\sum_{r\in\Z^n}
\Delta_{\B}(xq^r) \prod_{i=1}^n (-1)^{r_i} 
x_i^{(2n-1)r_i-i+1} q^{(2n-1)\binom{r_i}{2}} \\
=2(q)_{\infty}^n 
\prod_{i=1}^n \theta(x_i;q)
\prod_{1\leq i<j\leq n} \theta(x_i/x_j,x_ix_j;q),
\end{multline}
with $(q,x)\mapsto (q^{\kappa},y)$.

\medskip

Again \eqref{Eq_RR-A2n2a} can be understood representation-theoretically, but
this time the relevant Kac--Moody algebra is $\A_{2n}^{(2)}$.
According to \cite[Lemma 2.3]{BW13}
the right-hand side of \eqref{Eq_a2n2} with $\hat{x}$ interpreted not
as $\hat{x}=(-x_1,\dots,-x_{n-1},-1)$ but as
\[
\hat{x}_i=\eup^{-\alpha_0-\cdots-\alpha_{n-i}}\;\;(1\leq i\leq n)
\]
and $q$ as
\begin{equation}\label{Eq_null}
q=\eup^{-2\alpha_0-\cdots-2\alpha_{n-1}-\alpha_n}
\end{equation}
is the $\A_{2n}^{(2)}$ character 
\[
\eup^{-m\La_n} \ch V(m\La_n).
\]
The substitution \eqref{Eq_A2n2-spec} corresponds to 
\begin{equation}\label{Eq_specA}
\eup^{\alpha_0}\mapsto -1 
\quad\text{and}\quad
\eup^{\alpha_i}\mapsto q\;\;(1\leq i\leq n).
\end{equation}
Denoting this by $F$, we have the general specialisation formula
\begin{multline}\label{Eq_principal}
F\big(\eup^{-\La} \ch V(\La)\big) 
=\frac{(q^{\kappa};q^{\kappa})_{\infty}^n}{(q)_{\infty}^n} 
\prod_{i=1}^n  \theta\big(q^{\la_1-\la_i+p+i};q^{\kappa}\big) \\ \times
\prod_{1\leq i<j\leq n} 
\theta\big(q^{\la_i-\la_j-i+j},q^{\la_i+\la_j-i-j+2n+1};q^{\kappa}\big),
\end{multline}
where $\kappa=2n+2\la_0+1$ and
\[
\La=2\la_n\La_0+(\la_{n-1}-\la_n)\La_1+\cdots+(\la_1-\la_2)\La_{n-1}+(\la_0-\la_1)\La_n
\]
for $\la=(\la_0,\la_1,\dots,\la_n)$ a partition.
For $\la=(m,0^n)$ (so that $\La=m\La_n$) this is in accordance with 
\eqref{Eq_RR-A2n2a}.

\subsection{Proof of \eqref{Eq_RR-A2n2b}}

In \eqref{Eq_interm} we set $x_n=q^{1/2}$ so that
\begin{align*}
\sum_{\substack{\la \\[1pt] \la_1\leq m}} &
q^{\abs{\la}} P'_{2\la}\big(x_1^{\pm},\dots,x_{n-1}^{\pm},q^{1/2};q\big) \\[-1mm]
&=\frac{1}{(q)_{\infty}^{n-1}(q^2)_{\infty} 
\prod_{i=1}^{n-1}  (q^{3/2}x_i^{\pm},qx_i^{\pm 2})_{\infty}
\prod_{1\leq i<j\leq n-1} (qx_i^{\pm} x_j^{\pm})_{\infty}} \\[1mm] 
&\qquad \times
\sum_{r_1,\dots,r_{n-1}=-\infty}^{\infty} \sum_{r_n=0}^{\infty}
\frac{\Delta_{\C}(\hat{x} q^r)}{\Delta_{\C}(\hat{x})}\,
\prod_{i=1}^n (-1)^{r_i} \hat{x}_i^{\kappa r_i}
q^{\frac{1}{2}\kappa r_i^2-nr_i},
\end{align*}
where $\kappa=2m+2n+1$ and $\hat{x}=(x_1,\dots,x_{n-1},q^{1/2})$.
The $r_n$-dependent part of the summand is
\[
(-1)^{r_n} q^{\kappa\binom{r_n+1}{2}-nr_n}
\frac{1-q^{2r_n+1}}{1-q}\prod_{i=1}^{n-1}
\frac{x_iq^{r_i}-q^{r_n+1/2}}{x_i-q^{1/2}}\cdot
\frac{x_iq^{r_n+r_i+1/2}-1}{x_iq^{1/2}-1},
\]
which is readily checked to be invariant under the substitution
$r_n\mapsto -r_n-1$.
Hence
\begin{align*}
\sum_{\substack{\la \\[1pt] \la_1\leq m}} & 
q^{\abs{\la}} P'_{2\la}\big(x_1^{\pm},\dots,x_{n-1}^{\pm},q^{1/2};q\big) \\[-1mm]
&=\frac{1}{2(q)_{\infty}^n
\prod_{i=1}^{n-1} (-1) \theta(q^{1/2}x_i,x_i^2;q) 
\prod_{1\leq i<j\leq n-1} \theta(x_i/x_j,x_ix_j;q)} \\[1mm] 
&\qquad \times
\sum_{r\in\Z^n} 
\Delta_{\C}(\hat{x} q^r)
\prod_{i=1}^n (-1)^{r_i} \hat{x}_i^{\kappa r_i-i}
q^{\frac{1}{2}\kappa r_i^2-nr_i+\frac{1}{2}}.
\end{align*}
Our next step is to replace $x_i\mapsto x_{n-i+1}$ and
$r_i\mapsto r_{n-i+1}$. By $\theta(x;q)=-x\theta(x^{-1};q)$ and
\eqref{Eq_simp}, this leads to
\begin{align}\label{Eq_a2n2b}
\sum_{\substack{\la \\[1pt] \la_1\leq m}} & 
q^{\abs{\la}} P'_{2\la}\big(q^{1/2},x_2^{\pm},\dots,x_n^{\pm};q\big) \\[-1mm]
&=\frac{1}{(q)_{\infty}^n \prod_{i=1}^n  \theta(-q^{1/2}\hat{x}_i;q) 
\theta(\hat{x}_i^2;q^2) 
\prod_{1\leq i<j\leq n} \theta(\hat{x}_i/\hat{x}_j,\hat{x}_i\hat{x}_j;q)} 
\notag \\[1mm] 
&\qquad \times
\sum_{r\in\Z^n} 
\Delta_{\C}(\hat{x} q^r)
\prod_{i=1}^n (-1)^{r_i} \hat{x}_i^{\kappa r_i-i+1}
q^{\frac{1}{2}\kappa r_i^2-nr_i}, \notag 
\end{align}
where now $\hat{x}=(q^{1/2},x_2,\dots,x_n)$.
Again we are at the point where we can specialise, letting
\begin{equation}\label{Eq_A2n2b-spec}
q\mapsto q^{2n-1},\qquad x_i\mapsto q^{n-i+1/2}=:\hat{x_i}\;\; (1\leq i\leq n).
\end{equation}
This is consistent, since $x_1=q^{1/2}\mapsto q^{n-1/2}$. By 
\begin{multline*}
(q^{2n-1};q^{2n-1})_{\infty}^n 
\prod_{i=1}^n \theta(-q^{2n-i};q^{2n-1}) 
\theta(q^{2n-2i+1};q^{4n-2})  \\ \times
\prod_{1\leq i<j\leq n} \theta(q^{j-i},q^{2n-i-j+1};q^{2n-1})
=2(q)_{\infty}^n
\end{multline*}
this gives rise to
\[
\sum_{\substack{\la \\[1pt] \la_1\leq m}} 
q^{2\abs{\la}} P_{2\la}\big(1,q,q^2,\dots;q^{2n-1}\big) \\
=\frac{\M}{2(q)_{\infty}^n},
\]
where
\[
\M:=\sum_{r\in\Z^n} \Delta_{\C}(\hat{x} q^{(2n-1)r})
\prod_{i=1}^n (-1)^{r_i} \hat{x}_i^{\kappa r_i-i+1}
q^{\frac{1}{2}(2n-1)\kappa r_i^2-(2n-1)nr_i}.
\]
Expressing $\M$ in determinantal form using
\eqref{Eq_detC} yields
\begin{multline*}
\M=\det_{1\leq i,j\leq n} \bigg(\sum_{r\in\Z} (-1)^r \hat{x}_i^{\kappa r-i+1}
q^{\frac{1}{2}(2n-1)\kappa r^2-(2n-1)nr} \\ \times
\Big( (\hat{x}_i q^{(2n-1)r})^{j-1}-(\hat{x}_iq^{(2n-1)r})^{2n-j+1}\Big)\bigg).
\end{multline*}
We now replace $(i,j)\mapsto (j,i)$ and, viewing the resulting determinant 
as of the form $\det\big(\sum_r u_{ij;r}-\sum_r v_{ij;r}\big)$, we change 
the summation index $r\mapsto -r$ in the sum over $u_{ij;r}$. 
The expression for $\M$ we obtain is exactly \eqref{Eq_same} 
except that $(-1)^{i-j} q^{b_{ij}}$ is replaced by $q^{c_{ij}}$
and $y_i$ is given by $q^{n-i+1}$ instead of $q^{(\kappa+1)/2-i}$.
Following the previous proof results in \eqref{Eq_RR-A2n2b}. 

\medskip

To interpret \eqref{Eq_RR-A2n2b} in terms of $\A_{2n}^{(2)}$,
we note that by \cite[Lemma 2.2]{BW13} the right-hand side of 
\eqref{Eq_a2n2b} in which $\hat{x}$ is interpreted as 
\[
\hat{x}_i=-q^{1/2} \eup^{\alpha_0+\cdots+\alpha_{i-1}}\;\;(1\leq i\leq n)
\]
(and $q$ again as \eqref{Eq_null}) corresponds to the $\A_{2n}^{(2)}$ 
character 
\[
\eup^{-2m\La_0} \ch V(2m\La_0).
\]
The specialisation \eqref{Eq_A2n2b-spec} is then again consistent
with \eqref{Eq_specA}. From \eqref{Eq_principal} with $\la=(m^{n+1})$
the first product-form on the right of \eqref{Eq_RR-A2n2b} 
immediately follows.

By level-rank duality we can also identify \eqref{Eq_RR-A2n2b} 
as a specialisation of the $\A_{2m}^{(2)}$ character 
$\eup^{-2n\La_0} \ch V(2n\La_0)$.

\subsection{Proof of \eqref{Eq_RR-Dn}}

Our final proof is the most complicated of the four.
Once again we iterate \eqref{Eq_key}, but now both
$x_{n-1}$ and $x_n$ remain unpaired:
\begin{multline*}
\lim_{y_1\to x_1^{-1}}\dots \lim_{y_{n-2}\to x_{n-2}^{-1}}
L^{(0)}_m(x_1,y_1,\dots,x_{n-2},y_{n-2},x_{n-1},x_n) \\
=L^{(n-2)}_m(x_1,\dots,x_n).
\end{multline*}
Accordingly, if we replace 
$x\mapsto(x_1,y_1,\dots,x_{n-2},y_{n-2},x_{n-1},x_n)$ 
in \eqref{Eq_Cn-RS-m} (thereby changing the rank from $n$ to $2n-2$)
and take the $y_i\to x_i^{-1}$ limit
for $1\leq i\leq n-2$, we obtain
\begin{align*}
&\sum_{\substack{\la \\[1pt] \la_1\leq m}} q^{\abs{\la}} 
P'_{2\la}\big(x_1^{\pm},\dots,x_{n-2}^{\pm},x_{n-1},x_n;q\big) \\[-1mm]
&\quad=\frac{1}{(q)_{\infty}^{n-2}(qx_{n-1}^2,qx_{n-1}x_n,qx_n^2)_{\infty}} \\[1mm]
& \qquad \times 
\frac{1}{\prod_{i=1}^{n-2} 
(qx_i^{\pm 2},qx_i^{\pm}x_{n-1},qx_i^{\pm}x_n)_{\infty}
\prod_{1\leq i<j\leq n-2} (qx_i^{\pm} x_j^{\pm})_{\infty}} \\[1mm] 
&\qquad \times
\sum_{r\in\Z^n} \frac{\Delta_{\C}(x q^r)}{\Delta_{\C}(x)}\,
\prod_{i=1}^n \bigg(\frac{x_i^{\kappa}}{x_{n-1}x_n}\bigg)^{r_i}
q^{\frac{1}{2}\kappa r_i^2-(n-1)r_i}
\frac{(x_i x_{n-1},x_i x_n)_{r_i}}{(qx_i/x_{n-1},qx_i/x_n)_{r_i}},
\end{align*}
where $\kappa=2m+2n$. It is important to note that the summand vanishes unless
$r_{n-1}$ and $r_n$ are both nonnegative.
Next we let $(x_{n-1},x_n)$ tend to $(q^{1/2},1)$ using
\[
\lim_{(x_{n-1},x_n)\to (q^{1/2},1)} 
\frac{\Delta_{\C}(x q^r)}{\Delta_{\C}(x)}\,
\prod_{i=1}^n 
\frac{(x_i x_{n-1},x_i x_n)_{r_i}}{(qx_i/x_{n-1},qx_i/x_n)_{r_i}}
=\phi_{r_n} 
\frac{\Delta_{\B}(\hat{x} q^r)}{\Delta_{\B}(\hat{x})},
\]
with $\phi_r$ as in \eqref{Eq_phi} and 
$\hat{x}:=(-x_1,\dots,-x_{n-2},-q^{1/2},-1)$.
Hence
\begin{align*}
&\sum_{\substack{\la \\[1pt] \la_1\leq m}} 
q^{\abs{\la}} 
P'_{2\la}\big(x_1^{\pm},\dots,x_{n-2}^{\pm},q^{1/2},1;q\big) \\[-1mm]
&\quad=\frac{1}{(q)_{\infty}^{n-1}(q^{3/2};q^{1/2})_{\infty} 
\prod_{i=1}^{n-2}  (qx_i^{\pm};q^{1/2})_{\infty}
(qx_i^{\pm 2})_{\infty}
\prod_{1\leq i<j\leq n-2} (qx_i^{\pm} x_j^{\pm})_{\infty}} \\[1mm] 
&\qquad \times
\sum_{r_1,\dots,r_{n-2}=-\infty}^{\infty}
\sum_{r_{n-1},r_n=0}^{\infty} \phi_{r_n}
\frac{\Delta_{\B}(\hat{x} q^r)}{\Delta_{\B}(\hat{x})}\,
\prod_{i=1}^n \hat{x}_i^{\kappa r_i}
q^{\frac{1}{2}\kappa r_i^2-\frac{1}{2}(2n-1)r_i}.
\end{align*}
Since the summand (without the factor $\phi_{r_n}$)
is invariant under the variable change
$r_n\mapsto -r_n$ as well as the change 
$r_{n-1}\mapsto -r_{n-1}-1$, we can rewrite this as
\begin{align*}
\sum_{\substack{\la \\[1pt] \la_1\leq m}} &
q^{\abs{\la}} 
P'_{2\la}\big(x_1^{\pm},\dots,x_{n-2}^{\pm},q^{1/2},1;q\big) \\
&=\frac{1}{(q)_{\infty}^{n-1}(q^{1/2};q^{1/2})_{\infty}
\prod_{i=1}^n \theta(\hat{x}_i;q^{1/2})
\prod_{1\leq i<j\leq n} \theta(\hat{x}_i/\hat{x}_j,\hat{x}_i\hat{x}_j)} 
\notag \\[1mm] 
&\qquad \times
\sum_{r\in\Z^n}
\Delta_{\B}(\hat{x} q^r) \prod_{i=1}^n \hat{x}_i^{\kappa r_i-i+1}
q^{\frac{1}{2}\kappa r_i^2-\frac{1}{2}(2n-1)r_i}, \notag 
\end{align*}
where, once again, we have used \eqref{Eq_simp} to clean up the
infinite products.
Before we can carry out the usual specialisation we need to
relabel $x_1,\dots,x_{n-2}$ as $x_2,\dots,x_{n-1}$ and, 
accordingly, we redefine $\hat{x}$ as $(-q^{1/2},-x_2,\dots,-x_{n-1},-1)$. 
Then
\begin{align}\label{Eq_dn}
\sum_{\substack{\la \\[1pt] \la_1\leq m}} & q^{\abs{\la}} 
P'_{2\la}\big(q^{1/2},x_2^{\pm},\dots,x_{n-1}^{\pm},1;q\big) \\
&=\frac{1}{(q)_{\infty}^{n-1}(q^{1/2};q^{1/2})_{\infty}
\prod_{i=1}^n \theta(\hat{x}_i;q^{1/2})
\prod_{1\leq i<j\leq n} \theta(\hat{x}_i/\hat{x}_j,\hat{x}_i\hat{x}_j)} 
\notag \\[1mm] 
&\qquad \times
\sum_{r\in\Z^n}
\Delta_{\B}(\hat{x} q^r) \prod_{i=1}^n \hat{x}_i^{\kappa r_i-i+1}
q^{\frac{1}{2}\kappa r_i^2-\frac{1}{2}(2n-1)r_i}, \notag 
\end{align}
for $n\geq 2$.
We are now ready to make the substitutions
\begin{equation*}\label{Eq_Dn-spec}
q\mapsto q^{2n-2},\qquad x_i\mapsto q^{n-i}\;\; (2\leq i\leq n-1),
\end{equation*}
so that $\hat{x}_i:=-q^{n-i}$ for $1\leq i\leq n$.
By 
\begin{multline*}
(q^{2n-2};q^{2n-2})_{\infty}^{n-1}(q^{n-1};q^{n-1})_{\infty}\prod_{i=1}^n 
\theta(-q^{n-i};q^{n-1}) \\ \times
\prod_{1\leq i<j\leq n} \theta(q^{j-i},q^{2n-i-j};q^{2n-2})=
4(q^2;q^2)_{\infty}(q)_{\infty}^{n-1}
\end{multline*}
and \eqref{Eq_PpP} this results in
\[
\sum_{\substack{\la \\[1pt] \la_1\leq m}}
q^{2\abs{\la}} P_{2\la}\big(1,q,q^2,\dots;q^{2n-3}\big) =
\frac{\M}{4(q^2;q^2)_{\infty}(q)_{\infty}^{n-1}},
\] 
with $\M$ given by
\[
\M:=\sum_{r\in\Z^n}
\Delta_{\B}(\hat{x} q^{2(n-1)r}) \prod_{i=1}^n \hat{x}_i^{\kappa r_i-i+1}
q^{(n-1)\kappa r_i^2-(n-1)(2n-1)r_i}.
\]
By the $\B_n$ determinant \eqref{Eq_Bn-VdM},
\begin{multline*}
\M=\det_{1\leq i,j\leq n} 
\bigg(\sum_{r\in\Z} \hat{x}_i^{\kappa r-i+1}
q^{(n-1)\kappa r^2-(n-1)(2n-1)r} \\ \times
\Big( \big(\hat{x}_i q^{2(n-1)r}\big)^{j-1}-
\big(\hat{x}_iq^{2(n-1)r}\big)^{2n-j}\Big)\bigg).
\end{multline*}
By the same substitutions that transformed \eqref{Eq_begin} into
\eqref{Eq_end} we obtain
\begin{multline*}
\M=\det_{1\leq i,j\leq n} \bigg((-1)^{i-j} q^{b_{ij}} 
\sum_{r\in\Z} y_i^{2(n-1) r-i+1} q^{2(n-1)\kappa \binom{r}{2}} \\
\times
\Big(\big(y_iq^{\kappa r}\big)^{j-1}+
\big(y_i q^{\kappa r}\big)^{2n-j-1}\Big)\bigg),
\end{multline*}
where $y_i$ and $b_{ij}$ are as in \eqref{Eq_yb}.
Recalling the Weyl denominator formula for $\D_n$ \cite{Krattenthaler99}
\[
\frac{1}{2} \det_{1\leq i,j\leq n} \big(x_i^{j-1}+x_i^{2n-j-1}\big)=
\prod_{1\leq i<j\leq n} (x_i-x_j)(x_ix_j-1)=:\Delta_{\D}(x)
\]
we can rewrite $\M$ in the form
\[
\M=2\sum_{r\in\Z^n} \Delta_{\D}(xq^r)
\prod_{i=1}^n y_i^{2(n-1) r_i-i+1} q^{2(n-1)\kappa \binom{r_i}{2}}.
\]
Taking the $a_1,\dots,a_{2n-2}\to 0$, $a_{2n-1}\to 1$ and $a_{2n}\to -1$ 
limit in Gustafson's multiple $_6\psi_6$ summation for the affine
root system $\A_{2n-1}^{(2)}$ \cite{Gustafson89} leads to
the following variant of the $\D_n^{(1)}$ Macdonald 
identity\footnote{As in the $\B_n^{(1)}$ case, the actual $\D_n^{(1)}$ 
Macdonald identity contains the restriction 
$\abs{r}\equiv 0\pmod{2}$ on the sum over $r$.}
\[
\sum_{r\in\Z^n} \Delta_{\D}(xq^r) \prod_{i=1}^n 
x_i^{2(n-1)r_i-i+1} q^{2(n-1)\binom{r_i}{2}}
=2(q)_{\infty}^n 
\prod_{1\leq i<j\leq n} \theta(x_i/x_j,x_ix_j;q).
\]
This implies the claimed product form for $\M$ and completes our 
proof.

\medskip

Again \eqref{Eq_RR-Dn} has a simple representation-theoretic interpretation.
According to \cite[Lemma 2.4]{BW13}
the right-hand side of \eqref{Eq_dn} in which $\hat{x}$ is interpreted not
as $\hat{x}=(-q^{1/2},-x_1,\dots,-x_{n-1},-1)$ but as
\[
\hat{x}_i=\eup^{-\alpha_i-\cdots-\alpha_n}\;\;(1\leq i\leq n)
\]
and $q$ as
\[
q=\eup^{-2\alpha_0-\cdots-2\alpha_n}
\]
yields the $\D_{n+1}^{(2)}$ character 
\[
\eup^{-2m\La_0} \ch V(2m\La_0).
\]
The specialisation \eqref{Eq_Dn-spec} then corresponds to
\[
\eup^{\alpha_0},\,\eup^{\alpha_n}\mapsto -1 
\quad\text{and}\quad
\eup^{\alpha_i}\mapsto q\;\;(2\leq i\leq n-1).
\]
Denoting this by $F$, we have
\begin{multline*}
F\big(\eup^{-\La} \ch V(\La)\big)
=\frac{(q^{\kappa};q^{\kappa})_{\infty}^n}
{(q^2;q^2)_{\infty}(q)_{\infty}^{n-1}}  \\ \times
\prod_{1\leq i<j\leq n} 
\theta\big(q^{\la_i-\la_j-i+j},q^{\la_i+\la_j-i-j+2n+1};q^{\kappa}\big),
\end{multline*}
where $\kappa=2n+2\la_0$ and
\[
\La=2(\la_0-\la_1)\La_0+(\la_1-\la_2)\La_1+\cdots+(\la_{n-1}-\la_n)\La_{n-1}+2\la_n\La_n,
\]
for $\la=(\la_0,\la_1,\dots,\la_n)$ a partition or half-partition 
(i.e., all $\la_i\in\Z+1/2$).
For $\la=(m,0^n)$ this agrees with \eqref{Eq_RR-Dn}.

\section{Proof of Theorem~\ref{Thm_Main4}}\label{Sec_Pf2}

For $k$ and $m$ integers such that $0\leq k\leq m$ we write
the near-rectangular partition $(\underbrace{m,\dots,m}_{r \text{ times}},k)$ 
as $(m^r,k)$.
\begin{theorem}[$\A_{n-1}^{(1)}$ RR and AG identities]\label{Thm_Main5}
Let $m$ and $n$ be positive integers and $k$ a nonnegative integer not
exceeding $m$. Then
\begin{multline}\label{Eq_limk}
\lim_{r\to\infty}
q^{-m\binom{r}{2}-kr} 
Q_{(m^r,k)}(1,q,q^2,\dots;q^n) \\
=\frac{(q^n;q^n)_{\infty} (q^{\kappa};q^{\kappa})_{\infty}^{n-1}}
{(q)_{\infty}^n}
\prod_{i=1}^{n-1} \theta(q^{i+k};q^{\kappa})
\prod_{1\leq i<j\leq n-1} \theta(q^{j-i};q^{\kappa}),
\end{multline}
where $\kappa=m+n$.
\end{theorem}
For $k=0$ (or $k=m$) this yields Theorem~\ref{Thm_Main4}.
Before we give a proof of the above theorem we remark that by a
similar calculation it also follos that
\begin{multline*}
\lim_{r\to\infty}
q^{-m\binom{r+1}{2}} 
Q_{(k,m^r)}(1,q,q^2,\dots;q^n) \\
=\qbin{k-m+n-1}{n-1}_q
\frac{(q^n;q^n)_{\infty} (q^{\kappa};q^{\kappa})_{\infty}^{n-1}}
{(q)_{\infty}^n}
\prod_{1\leq i<j\leq n} \theta(q^{j-i};q^{\kappa}),
\end{multline*}
for $k\geq m$. 

\begin{proof}[Proof of Theorem~\ref{Thm_Main5}]
The following identity for the modified Hall--Littlewood polynomials
indexed by near-rectangular partitions is a special case of 
\cite[Corollary 3.2]{BW13}:
\begin{multline*}
Q'_{(m^r,k)}(x;q)=(q)_r(q)_1
\sum_{\substack{u\in\NN^n\\ \abs{u}=r+1}}
\sum_{\substack{v\in\NN^n\\ \abs{v}=r}}
\prod_{i=1}^n x_i^{ku_i+(m-k)v_i} q^{k\binom{u_i}{2}+(m-k)\binom{v_i}{2}}  
\\ \times
\prod_{i,j=1}^n 
\frac{(qx_i/x_j)_{u_i-u_j}}{(qx_i/x_j)_{u_i-v_j}}\cdot
\frac{(qx_i/x_j)_{v_i-v_j}}{(qx_i/x_j)_{v_i}}.
\end{multline*}
It suffices to compute the limit on the left-hand side of \eqref{Eq_limk}
for $r$ a multiple of $n$. Hence we replace $r$ by $nr$ in the above
expression, and then shift $u_i\mapsto u_i+r$ and
$v_i\mapsto v_i+r$ for all $1\leq i\leq n$, to obtain
\begin{align*}
Q'_{(m^{nr},k)}(x;q)&=(x_1\cdots x_n)^{mr}
q^{mn\binom{r}{2}+kr} (q)_{nr}(q)_1 \\ 
& \quad \times
\sum_{\substack{u\in\Z^n\\ \abs{u}=1}}
\sum_{\substack{v\in\Z^n\\ \abs{v}=0}}
\prod_{i=1}^n x_i^{ku_i+(m-k)v_i} 
q^{k\binom{u_i}{2}+(m-k)\binom{v_i}{2}}  \\ 
& \qquad\qquad\qquad \times \prod_{i,j=1}^n 
\frac{(qx_i/x_j)_{u_i-u_j}}{(qx_i/x_j)_{u_i-v_j}}\cdot
\frac{(qx_i/x_j)_{v_i-v_j}}{(qx_i/x_j)_{r+v_i}}.
\end{align*}
Since the summand vanishes unless $u_i\geq v_i$ for all $i$ and
$\abs{u}=\abs{v}+1$ it follows that $u=v+\epsilon_{\ell}$, for
some $k=1,\dots,n$, where $(\epsilon_{\ell})_i=\delta_{\ell i}$.
Hence
\begin{align*}
Q'_{(m^{nr},k)}(x;q)&=(x_1\cdots x_n)^{mr}
q^{mn\binom{r}{2}+kr} (q)_{nr} \\ 
&\quad \times \sum_{\substack{v\in\Z^n\\ \abs{v}=0}}
\prod_{i=1}^n x_i^{mv_i} q^{m\binom{v_i}{2}} 
\prod_{i,j=1}^n \frac{(qx_i/x_j)_{v_i-v_j}}{(qx_i/x_j)_{r+v_i}} \\
& \qquad \qquad \times 
\sum_{\ell=1}^n \big(x_{\ell}q^{v_{\ell}}\big)^k 
\prod_{\substack{i=1 \\ i\neq k}}^n \frac{1}{1-q^{v_i-v_{\ell}} x_i/x_{\ell}}.
\end{align*}
Next we use
\[
\prod_{i,j=1}^n (qx_i/x_j)_{v_i-v_j}
=\frac{\Delta(xq^v)}{\Delta(x)}\,
(-1)^{(n-1)\abs{v}} q^{-\binom{\abs{v}}{2}}
\prod_{i=1}^n x_i^{nv_i-\abs{v}} q^{n\binom{v_i}{2}+(i-1)v_i},
\]
where $\Delta(x):=\prod_{1\leq i<j\leq n} (1-x_i/x_j)$,
and
\[
\sum_{\ell=1}^n x_{\ell}^k 
\prod_{\substack{i=1 \\ i\neq k}}^n \frac{1}{1-x_i/x_{\ell}}=
\sum_{1\leq i_1\leq i_2\leq\cdots\leq i_k\leq n} x_{i_1}x_{i_2}\cdots x_{i_k}
=h_k(x)=s_{(k)}(x),
\]
where $h_k$ and $s_{\la}$ are the complete symmetric and Schur function,
respectively. Thus
\begin{multline*}
Q'_{(m^{nr},k)}(x;q)=(x_1\cdots x_n)^{mr} q^{mn\binom{r}{2}+kr}(q)_{nr} \\
\times 
\sum_{\substack{v\in\Z^n\\ \abs{v}=0}} s_{(k)}(xq^v)
\frac{\Delta(xq^v)}{\Delta(x)}
\prod_{i=1}^n x_i^{\kappa v_i}q^{\frac{1}{2}\kappa v_i^2+iv_i}
\prod_{i,j=1}^n \frac{1}{(qx_i/x_j)_{r+v_i}},
\end{multline*}
where $\kappa:=m+n$.
Note that the summand vanishes unless $v_i\geq -r$ for all $i$.
This implies the limit
\begin{multline*}
\lim_{r\to\infty}
q^{-mn\binom{r}{2}-kr} \frac{Q'_{(m^{nr},k)}(x;q)}
{(x_1\cdots x_n)^{mr}} \\
=\frac{1}{(q)_{\infty}^{n-1}\prod_{1\leq i<j\leq n} \theta(x_i/x_j;q)} 
\sum_{\substack{v\in\Z^n\\ \abs{v}=0}}
s_{(k)}(xq^v)
\Delta(xq^v) \prod_{i=1}^n x_i^{\kappa v_i}q^{\frac{1}{2}\kappa v_i^2+iv_i}.
\end{multline*}
The expression on the right is exactly the 
Weyl--Kac formula for the level-$m$ $\A_{n-1}^{(1)}$ character \cite{Kac90}
\[
\eup^{-\La} \ch V(\La) , \quad \La=(m-k)\La_0+k\La_1,
\]
provided we identify
\[
q=\eup^{-\alpha_0-\alpha_1-\cdots-\alpha_{n-1}}
\quad\text{and}\quad
x_i/x_{i+1}=\eup^{-\alpha_i}\;\; (1\leq i\leq n-1).
\]
Hence 
\[
\lim_{r\to\infty}
q^{-mn\binom{r}{2}-kr} \frac{Q'_{(m^{nr},k)}(x;q)}
{(x_1\cdots x_n)^{mr}}
=\eup^{-\Lambda} \ch V(\Lambda),
\]
with $\La$ as above. 
For $m=1$ and $k=0$ this was obtained in \cite{Kirillov00} by 
more elementary means.
The simultaneous substitutions $q\mapsto q^n$ and $x_i\mapsto q^{n-i}$
correspond to the principal specialisation \eqref{Eq_PS}. 
From \eqref{Eq_Kac} we can then read off the product form claimed
in \eqref{Eq_limk}.
\end{proof}

\bibliographystyle{amsplain}

\end{document}